\documentclass{article}
\usepackage{amssymb, latexsym, mathtools, amsthm, amsmath, amsfonts, stmaryrd, mathrsfs}
\usepackage{enumerate, xcolor, parskip, pdfsync}
\usepackage{authblk}
\usepackage[numbers]{natbib}
\theoremstyle{plain}

\renewenvironment{abstract}
{ \normalsize
\list{}{\setlength{\leftmargin}{.0cm}%
\setlength{\rightmargin}{\leftmargin}}%
\item {\bf \abstractname.}\relax}
{\endlist}
\newtheorem{theorem}{Theorem}[section]
\newtheorem{lemma}[theorem]{Lemma}
\newtheorem{prop}[theorem]{Proposition}
\newtheorem{corollary}[theorem]{Corollary}
\theoremstyle{definition}
\newtheorem{definition}[theorem]{Definition}

\newcommand{\restr}{\upharpoonright}  %restriction
\newcommand{\un}{\uparrow} %undefined
\newcommand{\sqbrad}[2]{\{\hspace{0.03cm}{#1} : {#2}\hspace{0.03cm}\}}

\newcommand{\RAND}{\mathsf{ML}}

\newcommand{\parb}[1]{\big({#1}\big)}

\newcommand{\dom}{\mathsf{dom}}

\newcommand{\pz}{\Pi^0_1}
\newcommand{\pzj}{\Pi^0_1(\emptyset')}

\newcommand{\pzt}{\Pi^0_2}

\newcommand{\twome}{\{0,1\}^{\omega}}

\newcommand{\zj}{\emptyset'}
\newcommand{\twomel}{\{0,1\}^{<\omega}}

\newcommand{\wedga}{\ \wedge\ }
\newcommand{\leqT}{\leq_T}
\newcommand{\geqT}{\geq_T}
\newcommand{\equivT}{\equiv_T}
\newcommand{\impl}{\ \implies\ }

\newcommand{\PA}{\mathsf{PA}}

\newcommand{\zjj}{\emptyset''}

\newcommand{\inv}{^{-1}}

\newcommand{\zjd}{\textup{\textbf{0}}'}
\newcommand{\zjjd}{\textup{\textbf{0}}''}

\title{Complexity of inversion of functions on the reals\thanks{Supported by Beijing Natural Science Foundation (IS24013).}}
\author{George Barmpalias, Mingyang Wang and Xiaoyan Zhang\thanks{Authors are in alphabetical order. 
Thanks to Liang Yu for corrections and suggestions.}}
\affil{State Key Lab of Computer Science\\ \vspace{0.1cm} 
Institute of Software, Chinese Academy of Sciences\\ \vspace{0.1cm}
University of  Chinese Academy of Sciences}

\begin{document}
\maketitle
\begin{abstract}
We study the complexity of deterministic and probabilistic inversions of partial computable functions on the reals.
\end{abstract}
\setcounter{tocdepth}{1}\tableofcontents
\section{Introduction}
\citet{Turing36}  laid the foundations for a general theory of computation which includes 
the computability of real functions \cite{AvigadBrattka2014}.
Early work \cite{Grzegorczyk1955, Freid_banach} 
cumulated in a robust  definition of
computable real functions based on effective continuity  and a
theory of computable analysis \cite{BrattkaHist2016, PR89, Weihrauchbook} which also 
 accommodates the study of computational complexity of real 
 functions \cite{KeriKobook}.\footnote{The  RAM-type model of \cite{BSS89}  is a different approach and not relevant to the present work.}
Besides  total real functions  the theory is just as robust
for partial functions, where {\sc while} loops can restrict their domain \cite{ZuckerPartial15}. 
Recent work  includes the study of extensions of  partial computable real functions to larger domains \cite{hoyrupExt17}.

The focus of this article is the complexity of inverting total and partial computable functions from {\em reals} (infinite binary sequences) to reals. 

Despite the vast research in computable analysis, this basic question is largely unexplored, 
with the exception of elementary facts and the recent work in \cite{Hoyrup14Irreversible}. This became apparent to us after 
a question of \citet{levin2023email} regarding the existence of  partial computable functions that preserve algorithmic randomness and 
are computationally hard to invert, even probabilistically.  

The finite analogues of such  {\em oneway functions} are fundamental in computational complexity while 
their existence remains a major open problem \cite{Levin2003}. 
So Levin's question was whether an answer can be obtained for computable real functions
without computational time and space constraints.

An affirmative answer was  given by \cite{barmpalias2024computable} and independently by  \citet{Gacs24oneway}. However
 the two results use a different formulations of probabilistic computation. 

Toward understanding Levin's oneway real functions and its  variants we study 
the complexity of inversion in terms of  computability focusing on:
\begin{enumerate}[(i)]
\item total and partial computable functions 
\item random-preserving and injective maps
\item deterministic and  probabilistic computation
\item null or positive (measure) domain and range.
\end{enumerate}
According to the standard framework,
partial computable real functions $f$  are continuous; the same is true if $f$ is partial computable relative to
an oracle. In this fashion  we only consider inversions by continuous  functions:  $g$ inverts $f$ on $y$
if $g(y)$ is in the domain of $f$ and $f(g(y))=y$.

Hardness of inversion could be due to the complexity of the domain or range.
Just as the range of a continuous $f$ may not be Borel, the range of a partial computable $f$ may not be arithmetically definable.
This aspect  is essential  and explains the role of measure in the hardness of inversion.

Needed background from computability and randomness
along with  complexity bounds  under  (i)-(iv) are given in \S\ref{basics}.
The role of  random oracles (probabilistic inversions)
on hardness of inversion is studied in \S\ref{SuyZFTF6y} along with
the cases where the domain or range is null.
We discuss Levin's notion of oneway functions 
and compare it with the variant considered by \citet{Gacs24oneway}. 

Under randomness-preservation (in the sense of \citet{MR0223179})  
injective functions are easier to invert. The existence of an injective partial computable oneway function
 is unknown.  A partial answer is given in \S\ref{injective} where a {\em left-oneway} (the notion of \citet{Gacs24oneway})  
 partial computable injective function is constructed.

{\bf Notation.}
%A \textit{real} is an infinite binary sequence and a \textit{string} is a finite binary sequence. 
We use $x,y,\cdots$ to denote reals and $\sigma,\tau,\cdots$ to denote strings. 
The  bit of $x$ at position $n$ is 
$x(n)$, starting with $n=0$. 
Let $|\sigma|$ be the length of string $\sigma$. Let 
$\omega$ be the set of natural numbers, 
$\{0,1\}^\omega$ be the set of all reals, $\{0,1\}^{<\omega}$ the set of finite binary strings and 
 $\{0,1\}^n$ the set of binary with length $n$. 
Let $\preceq$ be the prefix relation between strings and reals, and $\prec$ the strict prefix relation. 
Let $0^n$ be the $n$-bit string of 0s and  $0^\omega:=000\dots$.
For $x,y\in\twome$ let
\begin{itemize}
\item $x\upharpoonright_n$ be the $n$-bit prefix of $x$, so $x\upharpoonright_n=x(0)x(1)\cdots x(n-1)$,
\item $x\oplus y$ be the real $z$ with  $z(2n)=x(n)$ and $z(2n+1)=y(n)$.
\end{itemize}
A tree $T$ is a non-empty subset of $\{0,1\}^{<\omega}$ such that if $\sigma\prec\tau\in T$ 
then $\sigma\in T$. A  real $x$ is a \textit{path} of $T$ if $\forall n,\ x\upharpoonright_n\in T$ . Let $[T]$ be the class of all paths of $T$. 

The \textit{Cantor space} is $\{0,1\}^\omega$ with the topology generated by basic open sets
 \[
 \llbracket\sigma\rrbracket=\{x\in \{0,1\}^\omega:\sigma\prec x\}
 \]
for $\sigma\in \{0,1\}^{<\omega}$. 
The unions of basic open sets called are called
open subsets of $\{0,1\}^\omega$ and
their intersections are called 
are closed subsets.
Let $\mu$ be the  uniform measure on $\{0,1\}^\omega$ defined by $\mu(\llbracket\sigma\rrbracket)=2^{-|\sigma|}$. 
We  identify  $\{0,1\}^\omega$ with $\{0,1\}^\omega\times \{0,1\}^\omega$ via  $(x,y)\mapsto x\oplus y$.
Classes of reals with positive $\mu$-measure are \textit{positive} and those with $\mu$-measure $0$ are  \textit{null}. 

Let $x\leqT z$ denote that $x$ is computable by a Turing machine with oracle $z$.
The halting problem relative to $x$ is called the {\em jump} of $x$ and is denoted by $x'$.

We use the  notion of randomness and relative randomness in the sense of \citet{MR0223179}
and denote the class of all random reals by  $\mathsf{ML}$. We often refer to 
van Lambalgen's Theorem \cite{vLamb90} (see \cite[Corollary 6.9.3]{rodenisbook}): for every $x,y,z$
the real  $x\oplus y$ is $z$-random if and only if $x$ is $z$-random and $y$ is $x\oplus z$-random.

\section{Complexity bounds and  inversions}\label{basics}
We give  complexity bounds on the domain, range and inversions of partial computable functions,
assuming familiarity with the arithmetical, hyperarithmetical and analytic hierarchies
on $\omega, \twome$. We write $\pz(w)$ for $\pz$ relative to oracle $w$ and similarly for the other definability
classes.

\subsection{Partial computable functions on the reals}
A Turing machine with an infinite oracle tape and an infinite one-way output tape can be regarded as 
an effective map $f$ on the reals: given  $x\in\twome$ on its oracle tape 
the machine runs indefinitely and if the cells of the output tape contain the digits of $y$ we write $y=f(x)$.
This is the standard concept of computability on the reals from computable analysis \cite{PR89,WEIHRAUCH1993191}.

Let $f:\subseteq \{0,1\}^\omega\to \{0,1\}^\omega$ denote  a  function from a nonempty subset of $\{0,1\}^\omega$ to $\{0,1\}^\omega$.  
Let $\mathsf{dom}(f)$ be the domain and $f(\{0,1\}^\omega)$ be the range of $f$.

We write $f(x)\downarrow$ if $x\in\dom(f)$ and $f(x)\uparrow$ otherwise. 
We say that $f$ is {\em partial computable} if 
 there is an oracle Turing functional $\Phi$ such that 
\[
\big(f(x)\downarrow\iff \forall n,\ \Phi^x(n)\downarrow\big)\ \land\ \big(f(x)\downarrow\implies\forall n, \Phi^x(n)=f(x)(n)\big).
\]
If $\Phi$ has access to an additional oracle $w$ we say that $f$ is partial $w$-computable
and write $f\leqT w$.
Every $w$-computable function is continuous. 
Every continuous function with $G_\delta$ domain is $w$-computable for some oracle $w$. 

A \textit{representation} is a partial computable $\hat{f}:\subseteq \twomel\to\twomel$ with
\[
\parb{\sigma\prec\tau\wedga \hat{f}(\tau)\downarrow}\impl 
\parb{\hat{f}(\sigma)\downarrow\wedga \hat{f}(\sigma)\preceq\hat{f}(\tau)}
\]
and is a representation of  $f:\subseteq \{0,1\}^\omega\to \{0,1\}^\omega$ if  for all $x$ and
\[
\forall n,\ \hat{f}(x\upharpoonright_n)\prec f(x)
\hspace{0.3cm}\textrm{and}\hspace{0.3cm}
f(x)\downarrow\iff \lim_n|\hat{f}(x\upharpoonright_n)|=\infty.
\]
A function $f$ is  $w$-computable if  it has a $w$-computable representation. 

\subsection{Complexity of  domain and range}
If $f$ is partial computable its domain is effectively $G_{\delta}$, namely a $\pzt$ class.
The converse also holds, so the domain of a partial computable $f$ might not be closed.

\begin{prop}\label{domPi02}
If $f:\subseteq \{0,1\}^\omega\to \{0,1\}^\omega$ is  partial computable:
\begin{enumerate}[(i)]
\item the domain of $f$ is  a $\Pi^0_2$ class
\item  the restriction of $f$  to any $\Pi^0_2$ class $P$ is partial computable
\item every  $\Pi^0_2$ class is the domain of a partial computable $f$. 
\end{enumerate}
\end{prop}
\begin{proof}
Let $\hat{f}$ be a partial computable representation of $f$. 
Since
\[x\in\mathsf{dom}(f)\iff\forall n\ \exists m\ \exists s\ |\hat{f}_s(x\upharpoonright_m)\downarrow|\geq n\]
we get (i).
For (ii) let $(V_{i,j})$ be a computable family of clopen sets with
\[
P=\bigcap_i\bigcup_j V_{i,j}
\hspace{0.3cm}\textrm{and}\hspace{0.3cm}
V_{0,0}=\{0,1\}^\omega,  V_{i,j}\subseteq V_{i,j+1}.
\]
Define the partial computable function $\hat{g}(\sigma)=\hat{f}(\sigma)\upharpoonright_{t_{\sigma}}$ where
\[
t_{\sigma}:=\max\sqbrad{i}{i\leq|\hat{f}(\sigma)|\wedga \llbracket\sigma\rrbracket\subseteq V_{i,|\sigma|}}.
\]
If $\sigma\prec\tau$ and $\hat{g}(\tau)\downarrow$ then 
$\hat{f}(\tau)\downarrow$ and $\hat{f}(\sigma)\downarrow\preceq\hat{f}(\tau)$ and 
\[
\llbracket\tau\rrbracket\subset\llbracket\sigma\rrbracket\impl t_{\tau}\geq t_{\sigma}
\impl \hat{g}(\sigma)\downarrow\preceq\hat{g}(\tau)
\]
so $\hat{g}$ is  a representation of a partial computable $g$. 
If $x\in\mathsf{dom}(f)\cap P$ then 
\[
\forall i\ \exists k,j,\ \llbracket x\upharpoonright_k\rrbracket\subseteq V_{i,j}.
\]
If $n\geq k,j$, $|\hat{f}(x\upharpoonright_n)|\geq i$ by $\llbracket x\upharpoonright_n\rrbracket\subseteq V_{i,n}$ 
we get $|\hat{g}(x\upharpoonright_n)|\geq i$. Such $n$ exists for all $i$ so $g(x)\downarrow$. Since 
$\forall n,\ \hat{g}(x\upharpoonright_n)\preceq\hat{f}(x\upharpoonright_n)\prec f(x)$ we get $g(x)=f(x)$. 

If $x\notin\mathsf{dom}(f)$ then $g(x)\uparrow$. 
If $x\notin P$ then 
\[
\exists i\ \forall j,\ x\notin V_{i,j}\impl \forall n,\ |\hat{g}(x\upharpoonright_n)|\leq i\impl  g(x)\uparrow.
\] 
So $g$ is the restriction of  $f$  to $\mathsf{dom}(f)\cap P$. 
Finally (iii) follows from (ii).
\end{proof}

Continuous functions $f$ on subsets of $\twome$ map compact sets to compact sets.
Since $\twome$ is a Hausdorff space, the same holds for closed sets. 
An effective version of this fact is true when $f$ is partial computable.

Recall that if $P\in \Pi^0_1(w)$  there is a $w$-computable tree $\hat{P}$ with $P=[\hat{P}]$. 

\begin{prop}\label{rangePi01w}
Let $f:\subseteq \{0,1\}^\omega\to \{0,1\}^\omega$ be partial computable.

Then for any  $w\in\twome$:

\begin{enumerate}[(i)]
\item if $P\subseteq\mathsf{dom}(f)$ then $P\in \Pi^0_1(w)\impl f(P)\in\Pi^0_1(w)$, 
\item if $Q\in \Pi^0_1(w)$ there is $P\in \Pi^0_1(w)$ with  $f^{-1}(Q)=\mathsf{dom}(f)\cap P$. 
\end{enumerate}
\end{prop}
\begin{proof}
For (i) let $\hat{P}$ be a $w$-computable tree with $[\hat{P}]=P$. We show  
\begin{equation}\label{lK2gjbWdJx}
y\in f(P)\iff\forall n\ \exists\sigma\in\hat{P}\cap \{0,1\}^n,\ \hat{f}(\sigma)\prec y.
\end{equation}
If $y\in f(P)$ then  $\exists x\in P,\ f(x)=y$ so $\forall n\ \parb{x\upharpoonright_n\in\hat{P}\cap \{0,1\}^n\wedga\hat{f}(x\upharpoonright_n)\prec y}$.

For the converse, suppose that there is $\sigma_0,\sigma_1,\dots$ with $\sigma_i\in\hat{P}\cap \{0,1\}^i$ and $\hat{f}(\sigma_i)\prec y$. Since $\{0,1\}^\omega$ is compact, there is a subsequence of $\sigma_i$ (without loss of generality, the sequence $\sigma_i$ itself) that converges to some $x$. Let $\tau_i$ be the longest common prefix of $\sigma_i$, $x$. Since $\sigma_i\to x$, we get $|\tau_i|\to\infty$  so $\tau_i\to x$. Also
\begin{itemize}
\item since $\sigma_i\in\hat{P}$ we get $\tau_i\in\hat{P}$ so $x\in[\hat{P}]=P\subseteq\mathsf{dom}(f)$ and $f(x)\downarrow$,
\item since $\hat{f}(\tau_i)\preceq\hat{f}(\sigma_i)\prec y$ and $\tau_i\to x$ we get $f(x)=y$. 
\end{itemize}
This completes the proof of \eqref{lK2gjbWdJx} so $f(P)\in \Pi^0_1(w)$.
For (ii) we have 
\[
x\in f^{-1}(Q)\iff x\in\mathsf{dom}(f)\ \land\ \forall n\ \forall i\ \ 
\llbracket\hat{f}(x\upharpoonright_n)\rrbracket\cap Q_i\neq\emptyset
\] 
where each $Q_i$ is clopen, $i\mapsto Q_i$ is $w$-computable and $Q=\bigcap_i Q_i$. Then 
\[
P=\{x:\forall n\ \forall i\ \llbracket\hat{f}(x\upharpoonright_n)\rrbracket\cap Q_i\neq\emptyset\}
\] 
is $\Pi^0_1(w)$ and $f^{-1}(Q)=\mathsf{dom}(f)\cap P$. 
\end{proof}

Although the domain of any partial computable $f:\subseteq \{0,1\}^\omega\to \{0,1\}^\omega$ is arithmetically definable (indeed $\pzt$)
as a subset of $\twome$ it is possible that it is nonempty and has no arithmetically definable member.
\begin{prop}\label{xtZEktBm2g}
The following hold:
\begin{enumerate}[(i)]
\item if $x\in \Delta^1_1$ there is partial computable $f:\subseteq \{0,1\}^\omega\to \{0,1\}^\omega$ with domain $\dom(f)=\{z\}$ and $x<_T z$
\item there is partial computable $f:\subseteq \{0,1\}^\omega\to \{0,1\}^\omega$ with uncountable $\dom(f)$ and $\dom(f)\cap \Delta^1_1=\emptyset$.
\end{enumerate}
\end{prop}\begin{proof}
If $x\in \Delta^1_1$ there is an ordinal $\alpha<\omega_1^{\rm CK}$ 
with $x<_T \emptyset^{(\alpha)}$. 
Since $\{\emptyset^{(\alpha)}\}\in \Pi^0_2$ by
\cite[Theorem 2.1.4]{chong2015recursion}, 
by Proposition \ref{domPi02} (iii) we get  (i). Let
\[
\pi(\alpha):=0^{1+\alpha(0)} 1 0^{1+\alpha(1)}\dots 0^{1+\alpha(n)}\dots   
\]
be a computable injective function from $\omega^\omega$ into $\twome$.
By \cite[Proposition 2.5.2]{chong2015recursion} there is a computable tree $T\subseteq \omega^{\omega}$
such that $[T]$ is uncountable and $[T]\cap \Delta^1_1=\emptyset$.
Every $x\in \pi([T])$ computes a member of $[T]$ so $\pi([T])\cap \Delta^1_1=\emptyset$.

Since $\pi([T])\in \pzt$  we get (ii) by Proposition \ref{domPi02} (iii).  
\end{proof}
So  a partial computable $f:\subseteq \{0,1\}^\omega\to \{0,1\}^\omega$
may have a non-empty domain such that all arithmetically definable closed subsets 
(represented as the set of infinite paths through an arithmetically definable tree) of it are empty. 

By Proposition \ref{rangePi01w} (i) the range of any total computable $f:\twome\to\twome$ is  $\pz$.
However, the range of a partial computable $f$ is only $\Sigma^1_1$.
Recall (\cite[Definition 1.1.6]{chong2015recursion})  that a class
$Q\subseteq\twome$ is $\Sigma^1_1$  if for each $y$: 
\[
y\in Q\iff \exists\alpha\in\omega^\omega\ \forall s,\ P(\alpha, y, s)
\]
where  $P\subseteq \omega^\omega\times\twome\times\omega$ is a total computable predicate.

\begin{prop}[\cite{yu2024personal}]\label{qCNqs1lVbH}
The following hold:
\begin{enumerate}[(i)]
\item the range of any partial computable $f:\subseteq\twome\to\twome$ is $\Sigma^1_1$
\item each $\Sigma^1_1$ class is the range of a partial computable $f:\subseteq\twome\to\twome$.
\end{enumerate}
\end{prop}\begin{proof}
Clause (i) is straightforward. For (ii), we fix an effective injection from $\omega^\omega$ to $\{0,1\}^\omega$, for example 
$\pi(\alpha):=0^{1+\alpha(0)} 1 0^{1+\alpha(1)}\dots 0^{1+\alpha(n)}\dots$.

For any $Q\in \Sigma^1_1$, let $P\subseteq \omega^\omega\times\twome\times\omega$ 
be a total computable predicate with \[y\in Q\iff\exists\alpha\in\omega^\omega\ \forall s,\ P(\alpha,y,s).\] 
Let $f$ be such that \[f(x\oplus y)=\begin{cases}
    y & x=\pi(\alpha)\text{ for some }\alpha\ \land\ \forall s,\ P(\alpha,y,s) \\
    \uparrow & \text{otherwise}
\end{cases}\]

Then $f$ is partial computable and $f(\twome)=Q$.
\end{proof}
When $f$ has positive domain it has an effective restriction with $\pzj$ range.

\begin{prop}\label{Pi010ds} 
Let $f:\subseteq \{0,1\}^\omega\to \{0,1\}^\omega$ be a partial computable function.

If $f$ has positive domain then:
\begin{enumerate}[(i)]
\item there is a positive $P\in \Pi^0_1(\emptyset')$, $P\subseteq \mathsf{dom}(f)$ with $f(P)\in \Pi^0_1(\emptyset')$
\item the  restriction of $f$ to $P$ is partial computable
\end{enumerate}
and $P$ consists entirely of 2-random reals.
\end{prop}
\begin{proof}
By \citet{Kurtz81} (see  \cite[Theorem 6.8.3]{rodenisbook}) every positive $\pzt$ class has a
positive $\Pi^0_1(\emptyset')$ subclass $P$. By intersecting $P$ with a sufficiently large $\pzj$ class of 2-randoms
$P$ can be chosen to be a class of  2-randoms. So (i) follows by
Proposition \ref{domPi02} (i) and Proposition \ref{rangePi01w} (i).
Proposition \ref{domPi02} (ii) gives (ii).
\end{proof}

By Proposition \ref{Pi010ds} the complexity
bound in Proposition \ref{qCNqs1lVbH} can be reduced by considering
partial computable restrictions of $f$.

\subsection{Complexity of inversion}
For $f, g:\subseteq \twome\to\twome$ we say that $g$ is
an {\em inversion} of $f$ if 
\[
\forall y\in f(\twome),\ f(g(y))=y.
\]
By \cite[Corollary 3.2]{barmpalias2024computable} every total computable  injective $f$ has
a total computable inversion.  
We are particularly interested in functions that preserve randomness.
\begin{definition}
We say that $f:\subseteq \{0,1\}^\omega\to \{0,1\}^\omega$ is 
\textit{random-preserving} if it has positive domain and 
$f(x)$ is random for each random $x\in \dom(f)$.
\end{definition}
By \cite[Theorem 4.3]{milleryutran}, if $x$ is $w$-random and $y\leq_T x$ is random then $y$ is $w$-random. 
So partial computable random-preserving functions preserve $w$-randomness for all $w$. 

\begin{theorem}
There is a total computable random-preserving surjection $f$ which has no continuous inversion. 
\end{theorem}\begin{proof}
Let $(z'_s)$ be the universal enumeration of the jump $z'$ of $z$ and set
\begin{align*}
p^z(\langle n,s\rangle)\ &:=\ \begin{cases}
\ \ 2n & \textrm{if $n\in z'_{s+1}-z'_s$} \\[0.2cm]
2\langle n,s\rangle+1 & \text{otherwise}
\end{cases} \\[0.2cm]
h^z(x)(\langle n,s\rangle)\ &:=\ x(p^z(\langle n,s\rangle))
\end{align*}

By the relativization of \cite[Theorem 4.5]{barmpalias2024computable}:
\begin{enumerate}[(i)]
\item $h^z$ is a $z$-computable $z$-random preserving total surjection
\item if $w$ computes an inversion of $h^z$  then $w\geqT z'$.
\end{enumerate}
Since $z\mapsto h^z$ is effective, by \cite[Lemma 5.5]{barmpalias2024computable} 
the map
\[
f(x\oplus z):=h^z(x)\oplus z
\]
is a total computable random-preserving surjection. 
For contradiction, assume that $g$ is a continuous inversion of $f$. Since $f$ is total, $g$ has $G_\delta$ 
domain and therefore must be $w$-computable for some $w$. Now 
\[
\forall y,z:\ f(g(y\oplus z))=y\oplus z.
\] 
Let $g^z(y)$ be the even bits of $g(y\oplus z)$ so $g^z$ is an inversion of  $h^z$ and
\[
g^z\leqT z\oplus w
\hspace{0.3cm}\textrm{and}\hspace{0.3cm}
\forall z,\ z'\leq_T z\oplus w
\]
by (ii). In particular $w'\leq_T w\oplus w$ which is a  contradiction. 
\end{proof}
Partial inversions of effective maps  on the reals are more tractable. 
\begin{definition}
Given  $f, g:\subseteq \twome\to\twome$ we say that $g$ is a
\begin{itemize}
\item  {\em positive inversion} of $f$ if $\mu(\sqbrad{y}{f(g(y))=y})>0$.
\item  {\em probabilistic inversion} of $f$ if $\mu(\sqbrad{y\oplus r}{f(g(y\oplus r))=y})>0$.
\end{itemize}
\end{definition}
By \cite[Theorem 4.5]{barmpalias2024computable} there is a 
total computable random-preserving surjection $f$ such that every oracle that 
computes a probabilistic inversion of $f$ computes $\zjd$.
We show that this complexity bound is optimal.

\begin{theorem}\label{Sv7ueQvOq6}
Each total computable $f:\twome\to\twome$ with positive range has a $\zjd$-computable positive inversion.
\end{theorem}\begin{proof}
Clearly $f\inv(y)\in \pz(y)$ uniformly in $y$: $\sqbrad{x\oplus y}{x\in f\inv(y)}\in \pz$.
%\[
%\sqbrad{x\oplus y}{x\in f\inv(y)}\in \pz.
%\] 

So there exists a Turing functional $\Phi$ such that 
\[
\forall y\in f(\twome):\ \Phi(y')\in f\inv(y).
\]
By \cite[Theorem 8.14.5]{rodenisbook}
we have $y'\leqT y\oplus \zj$ for almost every $y$. 
Since $f(\twome)$ is positive, by countable additivity there exists a Turing functional $\Psi$ with 
\[
\mu(L)>0
\hspace{0.2cm}\textrm{where}\hspace{0.3cm}
L:=\sqbrad{y\in f(\twome)}{\Psi(y\oplus \zj)=y'}.
\]
Let $g(y)=\Phi(\Psi(y\oplus\zj))$ so $g\leqT \zjd$ and $g$ inverts $f$ on
\[
\sqbrad{y\in f(\twome)}{g(y)\in f\inv(y)}=\sqbrad{y}{f(g(y))=y}\supseteq L.
\]
Since $\mu(L)>0$ this is a positive subset of $f(\twome)$.
\end{proof}
We can now characterize the complexity of positive and probabilistic inversions of total computable
functions $f:\twome\to\twome$ with positive range.
\begin{corollary}\label{smjFFKVqRN}
The following are equivalent for $w\in\twome$:
\begin{enumerate}[(i)]
\item $w$ computes a positive inversion of every total computable $f$
\item $w$ computes a probabilistic inversion of every total computable $f$
\item $w$ computes $\zjd$
\end{enumerate}
where $f:\twome\to\twome$ is assumed to have positive range.
\end{corollary}\begin{proof}
By \cite[Theorem 4.5]{barmpalias2024computable} mentioned above we get (ii)$\to$(iii) while (i)$\to$(ii)
is immediate. Theorem \ref{Sv7ueQvOq6} gives (iii)$\to$(i).
\end{proof}
It is not hard to show that $\zjd$ does not compute a positive inversion of some  partial computable $f:\twome\to\twome$
with positive range.
We show that under a mild condition on $f$ oracle $\zjjd$ is sufficient for this purpose. 
The condition is that $f$ maps to random reals with positive probability, which is weaker than randomness-preservation and
present in Levin's  oneway functions \cite{barmpalias2024computable}.

\begin{theorem}\label{smaller than 0''}
Every  partial computable $f$ with $\mu(f^{-1}(\mathsf{ML}))>0$ has a 
positive inversion which is computable in $\zjjd$.
\end{theorem}
\begin{proof}
By Proposition \ref{Pi010ds} (ii)  we may assume that $\dom(f)\in \pzj$. Then  
\begin{equation}\label{XLGMcemu83}
\sqbrad{x\oplus y}{x\in f\inv(y)}\in \pzj
\end{equation}
so $f\inv(y)\in \pz(y\oplus \zj)$ uniformly in $y$ and for some Turing functional $\Phi$:
\[
\forall y\in f(\twome):\ \Phi((y\oplus \zj)')\in f\inv(y).
\]
By \cite[Theorem 8.14.5]{rodenisbook} for almost all $y$ we have $(y\oplus \zj)'\equivT y''\equivT  y\oplus \zjj$.

Since $\mu(f^{-1}(\mathsf{ML}))>0$  there exists a 2-random $x$ such that $f(x)$ is random. By 
\cite[Theorem 4.3]{milleryutran} and $f(x)\leqT x$ 
we get that $f(x)$ is a 2-random.

Since $f(\twome)\in \pzj$ we have $\mu(f(\twome))>0$
and  by countable additivity  there exists a Turing functional $\Psi$ with
$\mu(L)>0$ where 
\[
L:=\sqbrad{y\in f(\twome)}{\Psi(y\oplus \zjj)=(y\oplus \zj)'}.
\]
Let $g(y)=\Phi(\Psi(y\oplus\zjj))$ so $g\leqT \zjjd$ and $g$ inverts $f$ on
\[
\sqbrad{y\in f(\twome)}{g(y)\in f\inv(y)}=\sqbrad{y}{f(g(y))=y}\supseteq L.
\]
Since $\mu(L)>0$ this is a positive subset of $f(\twome)$.
\end{proof}

\section{Deterministic and probabilistic inversion}\label{SuyZFTF6y}
A partial computable $f:\subseteq\twome\to\twome$ is 
{\em nowhere effectively invertible} if 
\begin{equation}\label{qdhCBuFxck}
\textrm{$x\not\leqT f(x)$\ \  for all $x\in\dom(f)$.}
\end{equation}
In \S\ref{GDROTpaAVG} we show that this basic hardness for inversion 
condition is not hard to satisfy, even in the presence of additional requirements.
However these functions are invertible with the help of a random oracle, in certain ways. 
Hardness with respect to probabilistic inversions is explored in \S\ref{we2V38B7n1}.

\subsection{Inversion without random oracles}\label{GDROTpaAVG}
The simplest example of a nowhere effectively invertible partial computable 
function is the projection $f:\ x\oplus z\mapsto x$ restricted to a positive $\Pi^0_1$ class $P$ of randoms.
However $f$ is {\em  a.e.\  probabilistically invertible} in the sense that
\begin{equation}\label{tlqpsxl19f}
 \mu(\sqbrad{y\oplus r}{f(g(y\oplus r))=y})=\mu(f(\{0,1\}^\omega))>0
\end{equation}
for some partial computable $g$.

\begin{prop}
There is a partial computable $f:\subseteq\twome\to\twome$  with
\begin{enumerate}[(i)]
\item $f$ is random-preserving and nowhere effectively invertible 
\item $f$ is  a.e.\  probabilistically invertible.
\end{enumerate}
\end{prop}
\begin{proof}
Let $f$ be the restriction of $x\oplus z\mapsto x$
to a positive $\Pi^0_1$ class $P$ of randoms. 
Then $f$ is partial computable, random-preserving and
\[
x\oplus z\in P\impl x\oplus z\not\leqT x
\]
by van Lambalgen's theorem, so $f$ is nowhere effectively invertible. 
By Fubini's theorem \eqref{tlqpsxl19f} holds for $g(y\oplus r):=y\oplus r$ so (ii) holds.
\end{proof}

We turn to examples with null (and nonempty) domain or range. 

\begin{prop}\label{k6y6bnVCO8}
There are partial computable $f, h:\subseteq\twome\to\twome$  with 
\begin{enumerate}[(i)]
\item $f$ is injective and no arithmetically definable $g$ can invert $f$ on any real
\item $\dom(h)$ is uncountable and no $g\in \Delta^1_1$ can invert $h$ on any real.
\end{enumerate}
\end{prop}\begin{proof}
By Proposition \ref{xtZEktBm2g} (i)
$\{\emptyset^{(\omega)}\}\in \Pi^0_2$ so by Proposition \ref{domPi02} (ii)
\[
f:\emptyset^{(\omega)}\mapsto 0^\omega
\hspace{0.3cm}\textrm{is partial computable.}
\]
Any $g$ that inverts $f$ on $y=0^\omega$ defines $\emptyset^{(\omega)}$ which is not arithmetical.

By Proposition \ref{xtZEktBm2g} (ii) let $P$ be an uncountable $\pzt$ class with no $\Delta^1_1$ members.
The restriction of $x\mapsto 0^{\omega}$ to $P$ meets the requirements of (ii).
\end{proof}

A nowhere effectively invertible partial computable $f:\subseteq\twome\to\twome$  with
\[
\textrm{positive \ $\dom(f)\subseteq\RAND$ \ \ \ and\ \ \  $f(\twome)\cap \RAND=\emptyset$}
\]
can be given by \citet{Kurtz81} (see \cite[Theorem 8.21.3]{rodenisbook})
who constructed $f$ with the above properties and every real in its range is 1-generic.
Since 1-generics do not compute randoms \cite[Theorem 8.20.5]{rodenisbook}, $f$ is nowhere effectively invertible.

We end with the symmetric case of null domain and positive range.
\begin{prop}\label{8zsKJ9qshV}
There is a partial computable $f:\subseteq\twome\to\twome$ with
\begin{enumerate}[(i)]
\item $f$ is nowhere effectively invertible 
\item $\dom(f)\cap\RAND=\emptyset$ and  $f(\twome)\subseteq \RAND$ and $\mu(f(\twome))>0$.
\end{enumerate}
\end{prop}\begin{proof}
Let $Q\neq\emptyset$ be a positive $\pzt$ class of  2-randoms and
\[
f(x\oplus z):\simeq 
\begin{cases}
z & \textrm{if $x=\zj$ and $z\in Q$}\\
\un &\textrm{otherwise.}
\end{cases}
\]  
Since $x=\zj$ is $\pzt$ by  Proposition \ref{domPi02} (ii) $f$ is  partial computable.
By the choice of $Q$ we get  (ii). Since 2-randoms do not compute $\zj$ we get (i).
 \end{proof}

\subsection{Inversion with random oracles}\label{we2V38B7n1}
In computational complexity oneway functions are effectively calculable finite maps 
that cannot be effectively inverted with positive probability.
Levin \cite{levin2023email} defined oneway functions $f:\subseteq \twome\to\twome$  
to be partial computable with 
\begin{equation}\label{kaL6mSYVwF}
\mu(f\inv(\RAND))>0
\hspace{0.3cm}\textrm{and}\hspace{0.3cm}
\mu(\sqbrad{y\oplus r}{f(g(y\oplus r))=y})=0
\end{equation}
for each partial computable $g$.
A total computable oneway function was constructed in \cite{barmpalias2024computable}.
\citet{Gacs24oneway} constructed a partial computable $f$ with 
\[
\mu(\dom(f))>0
\hspace{0.3cm}\textrm{and}\hspace{0.3cm}
\mu\parb{\sqbrad{x\oplus r}{f(g(f(x)\oplus r))=f(x)}}=0
\]
for each partial computable $g$.
This  differs from \eqref{kaL6mSYVwF} in that
probability refers to the domain rather than the range of $f$.
Another difference is  condition  $\mu(f\inv(\RAND))>0$ which is weaker than randomness preservation
but stronger than $\mu(\dom(f))>0$. This suggests
two  variants of probabilistic inversion.
\begin{definition}
Let $f:\subseteq \{0,1\}^\omega\to \{0,1\}^\omega$ be partial computable. If 
\[
\mu(\dom(f))>0
\hspace{0.3cm}\textrm{and}\hspace{0.3cm}
\mu\parb{\sqbrad{x\oplus r}{f(g(f(x)\oplus r))=f(x)}}=0.
\]
for each partial computable $g$ 
we say that $f$ is \textit{left-oneway}. If 
\[
\mu(f(\{0,1\}^\omega))>0
\hspace{0.3cm}\textrm{and}\hspace{0.3cm}
\mu(\{y\oplus r:f(g(y\oplus r))=y\})=0
\]
for each partial computable $g$ 
we say that $f$ is \textit{right-oneway}.
\end{definition}
Given the absence of $\mu(f^{-1}(\mathsf{ML}))>0$ in the above definition note that
\begin{equation}\label{FDWSHXy1Xn}
\mu(f^{-1}(\mathsf{ML}))>0\underset{\not\Leftarrow}{\Rightarrow} \parb{\mu(\dom(f))>0\wedga \mu(f(\{0,1\}^\omega))>0}.
\end{equation}
The first implication is trivial and by Proposition \ref{Pi010ds} we may assume that the domain and range of $f$ are $\pzj$. 
Since $\mu(f^{-1}(\mathsf{ML}))>0$ there is a 2-random $x$ such that $f(x)$ is random. By
\cite[Theorem 4.3]{milleryutran} we get that $f(x)$ is a 2-random member of $f(\{0,1\}^\omega)$
which is  $\pzj$ so $\mu(f(\{0,1\}^\omega))>0$.

The failure of the converse in \eqref{FDWSHXy1Xn} is demonstrated later in Theorem \ref{vlsrMOASDY}.

Before we separate the classes of left and right oneway functions we show that 
$\mu(f^{-1}(\mathsf{ML}))>0$  is  equivalent (up to restrictions) to randomness-preservation.

\begin{lemma}\label{randPresDomRange}
If $f:\subseteq \{0,1\}^\omega\to \{0,1\}^\omega$ is partial computable then
 $\mu(f^{-1}(\mathsf{ML}))>0$ if and only if $f$ extends a random-preserving partial computable function. 
\end{lemma}
\begin{proof}
Suppose that 
$f$ extends a random-preserving partial computable  $h$.
Then $h$ has a positive domain $P$ and by 
Proposition \ref{Pi010ds} we may assume that $P$ consists of randoms.
Since $h$ is random-preserving  $h^{-1}(\mathsf{ML})=P$ so 
\[
\mu(f^{-1}(\mathsf{ML}))\geq \mu(h^{-1}(\mathsf{ML})) =\mu(P) >0.
\]
For the converse assume that $\mu(f^{-1}(\mathsf{ML}))>0$ so  there is
\begin{itemize}
\item $Q\in \Pi^0_1$ of randoms with $\mu(f^{-1}(Q))>0$
\item  $P\in \pz$ of randoms with $\mu(P\cap f^{-1}(Q))>0$.
\end{itemize}
Let $g$ be the restriction  $f$ to $P\cap f^{-1}(Q)$.
The by Proposition \ref{rangePi01w} (ii) the class $P\cap f^{-1}(Q)$ is $\pzt$ so $g$ is 
partial computable by Proposition \ref{domPi02} (ii).
Since $P\cap f^{-1}(Q)$ is positive and $P, Q$ consist of randoms, $g$ is random-preserving.
\end{proof}

Assuming randomness-preservation right-oneway implies left-oneway.

\begin{theorem}\label{onewayEquiv}
Every  partial computable random-preserving 
right-oneway function $f$ is left-oneway.
\end{theorem}\begin{proof}
Suppose that $f$ is random-preserving so it has positive domain.

If $f$ is not left-oneway there is a partial computable function $g$ with 
\[
\mu(L)>0
\hspace{0.3cm}\textrm{where}\hspace{0.3cm}
L:=\{x\oplus r:f(g(f(x)\oplus r))=f(x)\}
\]  
To show that $f$ is not right-oneway it suffices to show that
\[
\mu(R)>0
\hspace{0.3cm}\textrm{where}\hspace{0.3cm}
R:=\{y\oplus r:f(g(y\oplus r))=y\}.
\]  
Since $\mu(L)>0$ there is a $2$-random $x\oplus r\in L$. 
By van Lambalgen's theorem $x$ is $r\oplus\emptyset'$ random. Since $f$ is random-preserving $f(x)$ is random
and since $f(x)\leqT x$ by \cite[Theorem 4.3]{milleryutran} it is also 
$r\oplus\emptyset'$-random. By van Lambalgen's theorem $f(x)\oplus r$ is a $2$-random member of the $\pzt$ class $R$ so $\mu(R)>0$. 
\end{proof}

The converse is not true, even for total computable random-preserving $f$.

\begin{theorem}
There is a random-preserving partial computable left-oneway $f$ which is not right-oneway. 
\end{theorem}\begin{proof}
Let $f$ be a random-preserving partial computable oneway function, 
which exists by \cite[Theorem 4.4]{barmpalias2024computable}. Let $\bar{f}$ be defined by 
\[
\bar{f}(x)=\begin{cases}
    0f(z) & \text{ if } x=0z \text{ for some }z \\
    1z & \text{ if } x=1(0^\omega\oplus z) \text{ for some }z \\
   \  \uparrow & \text{ otherwise}
\end{cases}
\] 
then $\bar{f}$ is random-preserving since $f$ is. The partial computable function \[1z\mapsto 1(0^\omega\oplus z)\] inverts $\bar{f}$ on all $y\in\llbracket 1\rrbracket$, so $\bar{f}$ is not one-way or right-oneway. 

One the other hand, any partial computable $\bar{g}$ such that \[\mu(\{x:\bar{f}(\bar{g}(\bar{f}(x)\oplus r)=\bar{f}(x)\})>0\] induces a partial computable $g$ such that \[\mu(\{x:f(g(f(x)\oplus r))=f(x)\})>0\] which is impossible since $f$ is left-oneway. Therefore $\bar{f}$ is left-oneway.
\end{proof}

We show that Theorem \ref{onewayEquiv} fails if we do not require randomness-preservation. 
We say that a real $x$ is \textit{incomplete} if $\emptyset'\not\leq_T x$.

\begin{theorem}\label{vlsrMOASDY}
There is a total computable  $f:\subseteq \{0,1\}^\omega\to \{0,1\}^\omega$ such that:

\begin{itemize}
\item  $f$ is a surjection with $f(\mathsf{ML})\cap\mathsf{ML}=\emptyset$
\item $f$ is right-oneway but not left-oneway.
\end{itemize}
%$\mu(f^{-1}(\mathsf{ML}))=0$.
\end{theorem}
\begin{proof}
Let $P=\bigcap_s P_s$ be the $\Pi^0_1$ class of $\mathsf{PA}$ reals where $P_s$ is clopen and $s\mapsto P_s$ is computable.
Define the total computable surjection $f:\twome\to\twome$ by
\[
f(x\oplus z)=\begin{cases}
\ \ z&\textrm{if $x\in P$}\\
z\restr_n  \ 0^{\omega}&\textrm{if $n=\min\sqbrad{s}{x\not\in P_s}$.}
\end{cases}
\]
If $x\oplus z$ is random, $x\not\in P$ so $f(x\oplus z)$ is computable.  So $f(\mathsf{ML})\cap\mathsf{ML}=\emptyset$. 

Almost all reals $y\in f(\twome)$ are incomplete randoms. For each random $y$ every 
$w\in f\inv(y)$ computes a $\PA$ real. 
So for almost all $r$ and $y\in  f(\twome)$:
\begin{itemize}
\item $y\oplus r$ is an incomplete random
\item each real in $f\inv(y)$ computes a $\PA$ real.
\end{itemize}
By \cite{MR2258713frank} incomplete randoms do not compute $\PA$ reals so
\[
\mu(\sqbrad{r}{\textrm{$r\oplus y$ computes a member of $f\inv(y)$}})=0
\]
for almost all $y\in  f(\twome)$. This shows  that $f$ is right-oneway. Since
\[
x\not\in P\implies f(x\oplus z)\ \text{ends with}\ 0^{\omega}
\] 
the set $f((\{0,1\}^\omega-P)\times \twome)$ is countable. 
Since $\mu(P)=0$ by countable additivity there is $y$ with $\mu(f^{-1}(y))>0$. 
For this $y$ we have
\[
\mu(\{v\oplus r:f(v)=f(r)=y\})=\mu(f^{-1}(y))^2>0 
\] 
so for the function $g(w,r):=r$ we get  
\[
\mu\left(\{v\oplus r:f(g(f(v)\oplus r))=f(v)\}\right)>0.
\] 
Since $g$ is computable this shows that $f$ is not left-oneway.
\end{proof}

\section{A left-oneway injective function}\label{injective}
%Let $V$ be a measurable class of reals. The \textit{density} of a real $x$ in $V$ is \[\lim_{n\to\infty}2^n\mu\left(\llbracket x\upharpoonright_n\rrbracket\right)\cap V,\] if the limit exists. The Lebesgue's density theorem \cite[Theorem 1.2.3]{rodenisbook} states that almost every real (i.e.\ except a null set) in $V$ has density $1$ in $V$. 
Virtually all examples of partial computable functions that are hard to invert we have seen 
are non-injective, with the exception of Proposition \ref{k6y6bnVCO8}. 
%By \cite[Corollary 3.2]{barmpalias2024computable} every injective 
%total computable $f$ has a computable inversion. 
Random oracles do not reduce the complexity of inverting injective maps.
\begin{lemma}
Let $f:\subseteq\twome\to\twome$ be a partial computable 
injective function with positive domain. If $x\not\leqT f(x)$ for almost all $x\in \dom(f)$ then $f$ is left-oneway. 
\end{lemma}\begin{proof}
Since $f$ is injective we have
\[
\{x\oplus r: f(g(f(x)\oplus r))=f(x)\}=\{x\oplus r:g(f(x)\oplus r)=x\}.
\]
For a contradiction assume that $f$ is not left-oneway so 
\[
\mu(\{x\oplus r:g(f(x)\oplus r)=x\})>0 
\]
for some partial computable $g$. Then there is $x\in \dom(f)$ with $x\not\leqT f(x)$ and
\[
\mu(\{r : g(f(x)\oplus r)=x\})>0 
\]
so $\mu(\{r : x\leqT f(x)\oplus r\})>0$ which contradicts $x\not\leqT f(x)$.
\end{proof}
This raises a question: are  all effective injective functions  easy to invert?

\begin{theorem}\label{onewayInjective}
There is a left-oneway partial computable injective function.
\end{theorem}

The proof is an effective measure-theoretic argument in a well-established framework 
for the construction of effective probabilistic maps \cite{typical}. It is of independent interest since 
similar results  \cite{Martin60, BIENVENU201764, paris77} do not give injective maps. 

\subsection{Outline of the proof}\label{5nKaXzECB2}
Our argument is self-contained but follows the
standard framework for constructing probabilistic maps, see  
\cite[\S 8.21]{rodenisbook} or \cite{typical, BIENVENU201764}. 

Let $(\Phi_e)$ be an effective enumeration of all Turing functionals.
We construct a partial computable injectiion $f$ with positive domain such that for all $e$:
\[
R_e\text{:}\ \  \mu(\{x:\Phi_e(f(x))\downarrow=x\})=0.
\] 
We construct a representation $\hat{f}$ of $f$  in stages. At stage $s$ the domain of the current approximation $\hat{f}_s$
of $\hat f$ is a finite tree where each node with a successor has two successors. 
The {\em leaf nodes of  $\hat f_s$} are the leaf nodes of its domain.

{\bf Satisfying $R_e$.}
Suppose that  $\hat{f}(\sigma)$ is defined and all $\hat{f}(\sigma\tau), \tau\neq\lambda$ are undefined. To
 satisfy $R_e$ in $\llbracket\sigma\rrbracket$ we need to 
 define $f(x)$ so that for almost all $x\in\llbracket\sigma\rrbracket$ 
 \[
 \Phi_e(f(x))\uparrow\ \ \vee\ \ \ \Phi_e(f(x))\neq x.
 \]

{\em Case 1.} If $\Phi_e(y)\uparrow$ for all $y\in f(\llbracket\sigma\rrbracket)$ then $R_e$ is automatically satisfied.

{\em Case 2.} If $\Phi_e(y)(|\sigma|)\downarrow=i$  for  some $y\in f(\llbracket\sigma\rrbracket)$ with  oracle-use $\rho\prec y$ then
\begin{equation}\label{iwJ2Qp9e8T}
\parb{\rho\prec f(x)\wedga x(|\sigma|)\neq i}\impl
\Phi_e(f(x))(|\sigma|)=\Phi_e(y)(|\sigma|)\neq x(|\sigma|)
\end{equation}
which means that  $R_e$ is satisfied for this $x$. So we
\begin{enumerate}[(a)]
\item  satisfy $R_e$  in $\llbracket\sigma(1-i)\rrbracket$ by defining $\hat{f}(\sigma(1-i))=\rho$
\item define $\hat{f}(\sigma i)$ arbitrarily and
try to satisfy $R_e$ with $\sigma i$ in place of $\sigma$. 
\end{enumerate}
By iterating  we satisfy  $R_e$ for almost all $x\in\llbracket\sigma\rrbracket$. 

{\bf Non-uniformity.} Unfortunately 
we cannot effectively distinguish which case is true. 
We deal with this  by
creating a branch $\sigma \tau$ above $\sigma$ for each $\tau$ of certain length $l+1$ 
and assign to them different guesses. 
Branches $\sigma 00^l$, $\sigma 10^l$:
\begin{itemize}
\item believe  case 2 and search for $\rho$ as in \eqref{iwJ2Qp9e8T}
\item if $\rho$ is found they define $\hat{f}(\sigma (1-i)0^l)\succ \rho$ satisfying $R_e$  in $\llbracket\sigma(1-i)0^l\rrbracket$.
\end{itemize}
The $\hat{f}$-value at $\sigma i0^l$ is defined arbitrarily 
and $R_e$ in $\llbracket\sigma i0^l\rrbracket$ is considered by the nodes above.
If the guess is wrong:
\begin{itemize}
\item  the search never stops and 
$f(x)\uparrow$ for $x\in\llbracket\sigma 00^l\rrbracket\cup \llbracket\sigma 10^l\rrbracket$
\item  $R_e$ is satisfied since $f(x)\uparrow$
and $f$ is undefined in 
$\llbracket\sigma 00^l\rrbracket\cup \llbracket\sigma 10^l\rrbracket$.
\end{itemize}
In the meantime the remaining branches above $\sigma$  believe  {\em Case 1} and  proceed to other requirements. 
If $\sigma 00^l$, $\sigma 10^l$ have the wrong guess there will be a stage
where this  belief is seen to be wrong. At this point, although
$\hat{f}$ will have been defined above the remaining branches out of $\sigma$, 
we will extend the latter appropriately and reconsider 
$R_e$ on these newly added leaves.

{\bf Markers.} The assignment of requirement $R_e$ to a $\llbracket\sigma\rrbracket$ is done by placing 
an \textit{$e$-marker} on node $\sigma$ as elaborated above. 
The marker \textit{acts} if the $\rho$ in \eqref{iwJ2Qp9e8T} is found.
At each stage we keep a record of the status of each requirement on each leaf node:
initially $\lambda$ is \textit{$e$-unattended} for each $e$ and
\begin{itemize}
\item when $\sigma$ receives an $e$-marker, $\sigma 00^l$, $\sigma 10^l$ become \textit{$e$-waiting} and the other branches are  \textit{temporarily $e$-satisfied}
\item when this $e$-marker acts, the $e$-waiting branch $\sigma (1-i)0^l$ becomes \textit{$e$-satisfied} and the other $e$-waiting branch along with the temporarily $e$-satisfied nodes extending $\sigma$ become $e$-unattended
\item in all other cases the status of the leaf node is the same as its parent. 
\end{itemize}

A node is \textit{inactive} if it is $e$-waiting for some  $e$, indicating that no definition of $\hat f$ is permitted above them
under this status. We ensure that
\[
f(x)\uparrow\iff
\textrm{$x\upharpoonright_n$ stays inactive forever for some $n$.}
\]

{\bf Positive domain.}
As we elaborated,  if an $e$-marker on $\sigma$ never acts:
\begin{itemize}
\item  $f$ becomes undefined on measure $\leq 2^{-l}$ relative to $\llbracket\sigma\rrbracket$ 
\item no other $e$-markers are placed above $\sigma$ until the $e$-marker on $\sigma$ acts
\end{itemize}
so the set of strings with $e$-marker that never acts is prefix-free. Thus each  $R_e$ can only cause $f$ undefined on  measure $\leq 2^{-l}$. By letting $l=e+k$ for $k:=2$: 
 \[
 \mu(\dom(f))\geq 1-\sum_{e}2^{-e-k}=1-2^{1-k}>0.
 \]

{\bf Injective.}
We make $f$ injective by ensuring that at any stage $s$:
 \begin{enumerate}[(i)]
\item for any active leaf nodes $\sigma\ |\ \tau$ we have $\hat{f}_s(\sigma)\ |\ \hat{f}_s(\tau)$,
\item if there is an $e$-marker on $\sigma$ then 
$f_{s+1}(\llbracket\sigma\rrbracket)\subseteq f_s(\llbracket\sigma\rrbracket)$ where 
\[
f_s(\llbracket\sigma\rrbracket):=
\bigcup\left\{\llbracket\hat{f}_s(\sigma \tau)\rrbracket:\sigma 
\tau\text{ is an active leaf node of }\hat{f}_s\right\}
\] 
so $f(\llbracket\sigma\rrbracket)\subseteq f_s(\llbracket\sigma\rrbracket)$. 
\end{enumerate}

To this end we  often \textit{extend $\hat{f}_s$ to $\rho$ naturally}, 
which means that we find the unique leaf node $\sigma\prec\rho$ in $\hat{f}_s$
 and let $\hat{f}_{s+1}(\sigma \tau)=\hat{f}_s(\sigma) \tau$ for all $\tau$ such that $\sigma\tau\preceq\rho$. 
 Note that (i), (ii) are preserved under natural extensions. 

\subsection{Construction and verification}
We divide the stages for the construction to 3 types: 
\begin{itemize}
\item \textit{initialization stages}  $3\langle e,t\rangle$ 
where we put new $e$-markers on $e$-unattended nodes and do the initialization from them
\item \textit{activation stages} $3\langle e,t\rangle+1$  where  $e$-markers  requiring action act
\item \textit{expansion stages} $3t+2$ where we expand $f$ at its active leaf nodes. 
\end{itemize}

{\bf Construction.}
Let $k=2$, $\hat{f}_0(\lambda)=\lambda$ and $\lambda$ be $e$-unattended for each $e$. 

{\em Stage $s=3\langle e,t\rangle$:} for each active $e$-unattended 
leaf node $\sigma$:
\begin{itemize}
\item put an $e$-marker on $\sigma$
\item for $\tau=\sigma 00^{e+k}$ and $\tau=\sigma 10^{e+k}$ set $\hat{f}_{s+1}(\sigma \tau)=\hat{f}_s(\sigma)$, 
define $\hat{f}_{s+1}$ on strings between $\sigma$ and $\sigma\tau$, and set $\sigma\tau$ $e$-waiting;
\item for all other $\tau$ with $|\tau|=e+k+1$ naturally extend $\hat{f}_s$ to $\sigma\tau$ and set $\sigma\tau$ temporarily $e$-satisfied. 
\end{itemize}

{\em Stage $s=3\langle e,t\rangle+1$:} for each $e$-marker on $\sigma$ with $|\sigma|=n$ that has not acted, find $\tau$ with $\hat{f}_s(\sigma\tau)\downarrow=\hat{f}_s(\sigma)\rho$ and $\Phi_e(\rho)(n)[s]\downarrow=i$. If  $\tau$ does not exist  proceed to the next stage. Otherwise  we may assume that $\sigma\tau$ is a leaf node of $\hat{f}_s$, since otherwise we can replace $\tau$ with an  extension of it on leaf nodes. We then 
set: 
\begin{itemize}
\item  $j=1-i$,  $\hat{f}_{s+1}(\sigma j0^{e+k})=\hat{f}_s(\sigma)\rho 00$ and  $\sigma j0^{e+k}$ be $e$-satisfied
\item  $\sigma i0^{e+k}$ be $e$-unattended, $f_{s+1}(\sigma\tau)=f_s(\sigma)\rho 1$, $\hat{f}_{s+1}(\sigma i0^{e+k})=\hat{f}_s(\sigma)\rho 01$ 
\item all temporarily $e$-satisfied leaf nodes of $\hat{f}_s$ extending $\sigma$ be $e$-unattended
\end{itemize}
and  declare that this $e$-marker has acted.\footnote{suffixing $\hat{f}_s(\sigma)\rho$ with $00,1,01$
serves satisfying (i), (ii) above.}

{\em Stage $s=3\langle e,t\rangle+2$:} for each active leaf node $\sigma$ naturally extend $\hat{f}_s$ to $\sigma 0,\sigma 1$. 
This ensures $f(x)\downarrow$ when 
$x$ does not have prefixes that stay inactive forever. 

\paragraph{Verification.} As we discussed in \S\ref{5nKaXzECB2} the domain of $f$ is positive.

By induction on $s$ we get (i), (ii) from \S\ref{5nKaXzECB2}.

\begin{lemma}
The function $f$ is injective.
\end{lemma}\begin{proof}
Assuming  $x\neq y$, $f(x), f(y)\downarrow$ we show $f(x)\neq f(y)$. 

Let $\chi_s$ be the unique leaf node of $\hat{f}_s$ with $\chi_s\prec x$ and similarly for $\gamma_s\prec y$. 

Since $x\neq y$ for  sufficiently large $s$ we have $\chi_s\ |\ \gamma_s$. So
if $\chi_s$ and $\gamma_s$ are both active at $s$ by  (i) we get $\hat{f}_s(\chi_s)\ |\ \hat{f}_s(\gamma_s)$ so $f(x)\neq f(y)$. 
For this reason we may assume that at most one of $\chi_s$, $\gamma_s$ is active at sufficiently large $s$. 

Without loss of generality pick  $s$ such that $\chi_s$ is active and $\gamma_s$ is inactive. 

Since $f(y)\downarrow$ there exists $t>s$ such that 
\[
\textrm{$\chi_t\succ\chi_s$ is inactive and $\gamma_t=\gamma_s$ is active.}
\]
Let $\sigma$ be the node with an $e$-marker that made $\chi_t$ inactive so 
\begin{itemize}
\item  $\chi_s\preceq\sigma\prec\chi_t$ and  $\chi_s\ |\ \gamma_s$, $\tau\ |\ \gamma_t$ for any leaf node 
$\tau\succ\sigma$ 
\item by  (i) we have $\hat{f}_t(\tau)\ |\ \hat{f}_t(\gamma_t)$ so 
$f_t(\llbracket\sigma\rrbracket)\cap\llbracket\hat{f}_t(\gamma_t)\rrbracket=\emptyset$.
\end{itemize}
By (ii) we have  $f(x)\in f(\llbracket\sigma\rrbracket)\subset f_t(\llbracket\sigma\rrbracket)$ 
and $f(y)\in\llbracket\hat{f}_t(\gamma_t)\rrbracket$, so $f(x)\neq f(y)$. 
\end{proof}
It remains to show that if $g=\Phi_e$ for some $e$ then $g$ inverts $f$ on measure $0$.

For a contradiction assume otherwise so $\mu(L)>0$ where
\[
L=\{x:\Phi_e(f(x))=x\}.
\]
By Lebesgue's density theorem, there is $x_0\in L$ with density $1$ in $L$. For each $s$ let $\chi_s$ 
be the unique leaf node of $\hat{f}_s$ with $\chi_s\prec x_0$. Then:

\begin{itemize}
\item if $\chi_s$ is $e$-satisfied  the $e$-status of $\chi_t$ remains $e$-satisfied forever; 
\item if $\chi_s$ is $e$-waiting, since $f(x_0)\downarrow$ there must be $t>s$ where $\chi_t=\chi_s$ becomes $e$-satisfied or $e$-unattended; 
\item if $\chi_s$ is $e$-unattended  there must be $t>s$ where an $e$-marker is put on some $\chi_t$ and $\chi_{t+1}$ is $e$-waiting or temporarily $e$-satisfied; 
\item if $\chi_s$ is temporarily $e$-satisfied, the $e$-status of $\chi_t$ might remain temporarily $e$-satisfied forever, or  might become $e$-unattended for some $t>s$. 
\end{itemize}

Therefore there are three cases for  the status of $\chi_s$ as $s\to\infty$. 

{\em Case 1.} $\chi_s$ is $e$-satisfied for all sufficiently large $s$. Then $R_e$ is satisfied on all $x\in\chi_s$ so $R_e$ is satisfied for $x_0$, contradicting  $x_0\in L$. 

{\em Case 2.} $\chi_s$ is temporarily $e$-satisfied for all sufficiently large $s$. Then the $e$-marker on some $\sigma\prec\chi_s$ never acts so the corresponding $\tau,\rho$ are never found and $R_e$ is satisfied in $\llbracket\sigma\rrbracket$.
Since $x_0\in \llbracket\sigma\rrbracket$ this contradicts $x_0\in L$.

{\em Case 3:} the $e$-status of $\chi_s$ as $s\to\infty$ loops indefinitely through
\begin{enumerate}[\hspace{0.3cm} (a)]
\item $e$-unattended
\item $e$-waiting or temporarily $e$-satisfied
\item $e$-unattended
\end{enumerate}
Transition  (a)$\to$(b) is caused by an $e$-marker on some $\sigma$ which will  act when transition (b)$\to$(c) occurs.
Each such action produces a set of relative measure $2^{-e-k-1}$ where $R_e$ is satisfied, so the density of $x_0$ in $L$ is 
$\leq 1-2^{-e-k-1}$, contradicting the hypothesis that $x_0$ has density $1$ in $L$. 

We conclude that  $\mu(L)=0$ so all requirements are met.

%\bibliographystyle{abbrvnat}
%\bibliography{owrand}

\end{document}